\documentclass[final,3p,times]{elsarticle}
\DeclareOldFontCommand{\sc}{\normalfont\scshape}{\@nomath\sc}
\usepackage{amsmath,amssymb,amsfonts,amsthm,mathrsfs,comment,mathdots,multirow,tikz}
\usepackage{array}
\usepackage{blkarray}
\usepackage{capt-of}
\usepackage{color}
\usepackage{colortbl}
\usepackage{fancyhdr}
\usepackage{fancyvrb}
\usepackage{float}
\usepackage[T1]{fontenc}
\usepackage{framed}
\usepackage{geometry}
\usepackage{graphicx}
\usepackage[latin1]{inputenc}
\usepackage{stmaryrd}
\usepackage{upquote}
\usepackage{url}
\usepackage{xspace}	
\usepackage{booktabs}
\usepackage{mathtools}
\usepackage{soul}
\usepackage{fancyhdr}
\DeclarePairedDelimiter{\ceil}{\lceil}{\rceil}
\definecolor{lightgray}{gray}{0.5}
\definecolor{gray}{gray}{0.7}
\definecolor{darkgreen}{rgb}{0,0.6,0.13}
\newcommand{\nc}{\newcommand}
\nc{\dsp}{\displaystyle}
\nc{\txt}{\textstyle}
\nc{\reff}[1]{(\ref{#1})}
\nc{\mrm}[1]{\mathrm{#1}}
\nc{\udl}[1]{\underline{#1}}
\nc{\ovl}[1]{\overline{#1}}
\nc{\al}{\underline{\boldsymbol{\alpha}}}
\nc{\la}{\underline{\boldsymbol{\lambda}}}
\nc{\llbr}{\llbracket}
\nc{\rrbr}{\rrbracket}
\nc{\lbr}{\lbrack}
\nc{\rbr}{\rbrack}
\nc{\N}{\mathbb{N}}
\nc{\Z}{\mathbb{Z}}
\nc{\D}{\mathbb{D}}
\nc{\Q}{\mathbb{Q}}
\nc{\R}{\mathbb{R}}
\nc{\C}{\mathbb{C}}
\nc{\T}{\mathbb{T}}
\nc{\Stwo}{\mathbb{S}^2}
\nc{\tld}[1]{\tilde{#1}}
\nc{\wtld}[1]{\widetilde{#1}}
\nc{\hu}{\hat{u}}
\nc{\wh}[1]{\widehat{#1}}
\nc{\Fbf}{\textbf{F}}
\nc{\Gbf}{\textbf{G}}
\nc{\Lbf}{\textbf{L}}
\nc{\Nbf}{\textbf{N}}
\nc{\Ibf}{\textbf{I}}
\nc{\Dbf}{\textbf{D}} 
\nc{\Tbf}{\textbf{T}}
\nc{\Jbf}{\textbf{J}} 
\nc{\Rbf}{\textbf{R}}   
\nc{\ph}{\varphi}
\nc{\NN}{\mathcal{NN}}
\nc{\OO}{\mathcal{O}}
\nc{\sumeven}{\sum_{k=-N/2}^{N/2}{\hspace{-0.3cm}}'{\;\,}}
\nc{\sumevenk}{\sum_{k=-n/2}^{n/2}{\hspace{-0.3cm}}'{\;\,}}
\nc{\sumevenj}{\sum_{j=-m/2}^{m/2}{\hspace{-0.3cm}}'{\;\,}}
\nc{\sumodd}{\sum_{k=-\frac{N-1}{2}}^{\frac{N-1}{2}}}
\nc{\sumoddl}{\sum_{l=-\frac{N-1}{2}}^{\frac{N-1}{2}}}
\nc{\cqfd}{~\hbox{\vrule width 2.5pt depth 2.5 pt height 3.5 pt}}
\nc{\rr}[1]{\textcolor{red}{#1}}
\nc{\bb}[1]{\textcolor{blue}{#1}}
\nc{\rrr}[1]{\textcolor{red}{\textit{new:}\,\,#1}}
\nc{\rrrr}[1]{\textcolor{red}{\textit{new (simpler):}\,\,#1}}
\nc{\bbb}[1]{\textcolor{blue}{\textit{old:}\,\,#1}}
\nc{\hm}[1]{\textcolor{red}{\textbf{HM:} #1}}
\nc{\hz}[1]{\textcolor{blue}{\textbf{HZ: #1}}}
\nc{\ra}[1]{\renewcommand{\arraystretch}{#1}}
\nc{\bs}[1]{\boldsymbol{#1}}
\nc{\wt}{\widetilde{\times}}
\nc{\wtt}{\widetilde{\varprod}}
\nc{\wpp}{\widetilde{p}}
\nc{\wff}{\widetilde{f}}
\nc{\wTT}{\widetilde{T}}
\nc{\wKK}{\widetilde{K}}
\newtheorem{theorem}{Theorem}[section]
\newtheorem{proposition}[theorem]{Proposition}
\newtheorem{definition}[theorem]{Definition}

\setulcolor{red}
\journal{Journal of Computational Mathematics}

\begin{document}

\begin{frontmatter}

\title{Deep ReLU networks overcome the curse of dimensionality for generalized bandlimited functions}

\author[HMQD]{Hadrien Montanelli}
\ead{hadrien.montanelli@gmail.com}

\author[HY]{Haizhao Yang}
\ead{haizhao@purdue.edu}

\author[HMQD]{Qiang Du}
\ead{qd2125@columbia.edu}

\address[HMQD]{Department of Applied Physics and Applied Mathematics, Columbia University, New York, United States}
\address[HY]{Department of Mathematics, Purdue University, Indiana, United States}

\begin{abstract}
We prove a theorem concerning the approximation of generalized bandlimited multivariate functions by deep ReLU networks 
for which the curse of the dimensionality is overcome.
Our theorem is based on a result by Maurey and on the ability of deep ReLU networks to approximate Chebyshev polynomials and
analytic functions efficiently.
\end{abstract}


\end{frontmatter}

%%%%%%%%%%%%%%%%%%%%%%%%%%%%%%%%%%%%%%%%%%%%%%%%%%%%%%%%%%%%%%%%%%%%%%%%%%%
\section{Introduction}

The curse of dimensionality is a inevitable issue in high-dimensional scientific computing.
Standard numerical algorithms whose cost is exponential in the dimension $d$ are prohibitive when $d$ is large.
As a mesh-free function parametrization tool, neural networks are believed to be a suitable approach to conquer the curse of dimensionality.
In this paper, we show that deep ReLU networks overcome the curse of dimensionality for \textit{generalized bandlimited functions}, which we shall define at the end of the introduction.
Let us first quickly review what networks are.

\textit{Shallow networks} are approximations $\wff_W$ of multivariate functions $f:\R^d\rightarrow\R$ of the form
\begin{equation}
\wff_W(\boldsymbol{x}) = \sum_{i=1}^W \alpha_i \sigma(\boldsymbol{w}_i\cdot\boldsymbol{x} + \theta_i),
\label{eq:shallow}
\end{equation}
\noindent for some \textit{activation function} $\sigma:\R\rightarrow\R$, weights $\alpha_i,\theta_i\in\R$, $\boldsymbol{w}_i\in\R^d$ and integer $W\geq1$.
Each operation $\sigma(\boldsymbol{w}_i\cdot \boldsymbol{x} + \theta_i)$ is called a \textit{unit}
and the $W$ units in \eqref{eq:shallow} form a \textit{hidden layer}; this is a special form of nonlinear approximation \cite{devore1989, devore1998}. 
\textit{Deep networks} are compositions of shallow networks and have several hidden layers, and each unit of each layer 
performs an operation of the form $\sigma(\boldsymbol{w}\cdot \boldsymbol{x} + \theta)$. 
Following Yarotsky \cite{yarotsky2017}, we allow connections between units in non-neighboring layers.
We define the \textit{depth} $L$ of a network as the number of hidden layers and the \textit{size} $W$ as the total number of units.
In practice, networks with depth $L=\OO(1)$ are considered shallow, while deep networks have typically $L\gg1$ layers.

Before the revolution of deep learning \cite{lecun2015}, most research concerned shallow networks with sigmoid activation functions. 
Nowadays, networks using the \textit{REctifier Linear Unit} (\textit{ReLU}) activation function $\sigma(x)=\max(0,x)$ have become the most popular tool, partly because sigmoid activation functions lead to severe gradient degeneracy during the optimization process. 
It was also shown in \cite{glorot2011} that deep ReLU networks produce sparsity that helps a wide range of machine learning applications; smooth activation functions, including smoothed ReLU functions, do not.
This is why we focus on ReLU networks in this paper.

The theory of approximating functions using shallow networks goes back to 1989 when Cybenko showed that any continuous functions can be approximated by shallow networks \cite{cybenko1989}, while Hornik, Stinchcombe and White proved a similar result for Borel measurable functions \cite{hornik1989}.
In the 1990s, the attention shifted to the {\textit{approximation power}}\footnote{For a real-valued function $f$ in $\R^d$ whose smoothness is characterized by some integer~$m\geq1$, and for some prescribed accuracy $\epsilon>0$, one shows that there exists a shallow network $\wff_W$ of size $W=W(d,m)$ that satisfies $\Vert f - \wff_W \Vert \leq \epsilon$ for some norm $\Vert\cdot\Vert$.}
of shallow networks \cite{mhaskar1993, mhaskar1996, mhaskar1992, pinkus1999}. 
Of particular interest was the absence of the curse of dimensionality in the approximation of functions with fast decaying Fourier coefficients \cite{barron1993}. 

Fast forward to the 2010s and the success of deep networks, one of the most important theoretical problems is to determine why and when deep networks can lessen or break the curse of dimensionality, especially for ReLU networks. One may focus on a particular set of functions which have a very special structure (such as compositional or polynomial), and show that for this particular set deep networks overcome the curse of dimensionality \cite{bach2017, cohen2016, eldan2016, liang2016, montanelli2019c, petersen2018, poggio2017, shaham2018, telgarsky2016}.
Alternatively, one may consider a function space that is more generic for multivariate approximation in high dimensions, such as Korobov spaces \cite{korobov1959}, and prove convergence results for which the curse of dimensionality is lessened \cite{montanelli2019a}.

In this paper, we may consider \textit{generalized bandlimited} functions $f:B=[0,1]^d\rightarrow\R$ of the form
\begin{align}
& f(\bs{x}) = \int_{\R^d}F(\bs{w})K(\bs{w}\cdot\bs{x})d\bs{w}, \quad\mrm{supp}\,F\subset [-M,M]^d,\quad M\geq 1, \label{eq:bandlimited}
\end{align}
for some square-integrable function $F:[-M,M]^d\rightarrow\C$ and analytic kernel $K:\R\rightarrow\C$. 
{This class of functions contains several examples of Reproducing Kernel Hilbert Spaces (RKHSs), including the space of bandlimited functions. The latter are ubiquitous in science and engineering.
In information theory, bandlimited signals are often used for analysis and representation after sampling. 
In scientific computing, after discretization, functions are bandlimited by the Nyquist--Shannon sampling theorem. 
Studying the approximation power of ReLU networks for bandlimited functions is particularly important for neural network-based scientific computing in high dimensions.}
In Section~\ref{sec:theorem}, we shall show that for any measure $\mu$ such functions can be approximated to accuracy $\epsilon$ in the $L^2(B,\mu)$-norm by deep ReLU networks of depth $L=\OO\left(\log_2^2\frac{1}{\epsilon}\right)$ and size $W=\OO\left(\frac{1}{\epsilon^2}\log_2^2\frac{1}{\epsilon}\right)$.

We review some properties of deep ReLU networks in Section \ref{sec:properties}, providing new proofs of existing results (Propositions \ref{prop:multiplication} and \ref{prop:monomial}), as well as presenting new results (Propositions \ref{prop:chebyshev} and \ref{prop:series}, Theorem \ref{thm:analytic}). In Section \ref{sec:theorem}, we recall an existing theorem (Theorem \ref{thm:maurey}), before proving our main theorem (Theorem \ref{thm:bandlimited}).

%%%%%%%%%%%%%%%%%%%%%%%%%%%%%%%%%%%%%%%%%%%%%%%%%%%%%%%%%%%%%%%%%%%%%%%%%%%
\section{Approximation properties of deep ReLU networks}\label{sec:properties}

The ability of deep ReLU networks to implement the multiplication of two real numbers with amplitude $M$ was proved by Yarotsky in \cite[Prop.~1]{yarotsky2017}. 
Liang and Srikant proved a similar result for $M=1$ using networks with rectifier linear as well as binary step units in \cite[Thm.~1]{liang2016}.
In the rest of the paper, ``with accuracy $\epsilon$'' or ``bounded'' should be understood in the $L^\infty$-norm, unless stated otherwise.

%%%%%%%%%%%%%%%%%%%%%%%%%%%%%%%%%%%%%%%%%%%%%%%%%%%%%%%%%%%%%%%%%%%%%%%%%%%
\begin{proposition}[Multiplication in two dimensions]\label{prop:yarotsky}
For any scalar $M\geq1$, $N\geq1$ and $0<\epsilon<1$, there is a deep ReLU network $\widetilde{\pi}$ with inputs $(x_1,x_2)\in[-M,M]\times[-N,N]$, that has depth
\begin{align*}
L=\OO\left(\log_2\frac{MN}{\epsilon}\right),
\end{align*}
and size
\begin{align*}
W=\OO\left(\log_2\frac{MN}{\epsilon}\right),
\end{align*} 
such that
\begin{align*}
\left\Vert\widetilde{\pi}(x_1,x_2) - x_1x_2\right\Vert_{L^\infty([-M,M]\times[-N,N])}\leq\epsilon.
\end{align*}

Equivalently, if the network has depth $L=\OO\left(\log_2\frac{1}{\epsilon}\right)$ and size $W=\OO\left(\log_2\frac{1}{\epsilon}\right)$, the approximation error satisfies $\left\Vert\widetilde{\pi}(x_1,x_2) - x_1x_2\right\Vert_{L^\infty([-M,M]\times[-N,N])}\leq MN\epsilon$.
\end{proposition} 

We generalize the proposition of Yarotsky to the $d$-dimensional case.

%%%%%%%%%%%%%%%%%%%%%%%%%%%%%%%%%%%%%%%%%%%%%%%%%%%%%%%%%%%%%%%%%%%%%%%%%%%
\begin{proposition}[Multiplication in $d\geq2$ dimensions]\label{prop:multiplication}
For any scalar $M\geq1$ and $0<\epsilon<1$, and any integer $d\geq 2$, 
there is a deep ReLU network $\widetilde{\Pi}$ with inputs $(x_1,\ldots,x_d)\in[-M,M]^d$, 
that has depth 
\begin{align*}
L=\OO\left(d\log_2\frac{d}{\epsilon} + d^2\log_2M\right),
\end{align*}
and size 
\begin{align*}
W=\OO\left(d\log_2\frac{d}{\epsilon} + d^2\log_2M\right),
\end{align*} 
such that
\begin{align*}
\left\Vert\widetilde{\Pi}(x_1,\ldots,x_d) - x_1\ldots x_d\right\Vert_{L^\infty([-M,M]^d)}\leq\epsilon.
\end{align*}
\end{proposition} 

\begin{proof}
Let $M\geq1$ and $0<\epsilon<1$ be two scalars, and $d\geq 2$ an integer. 
For any scalar $A\geq 1$ and $B\geq 1$, let us call $\widetilde{\pi}$ the network of Proposition \ref{prop:yarotsky} that implements the multiplication $xy$, $x\in[-A,A]$, $y\in[-B,B]$, with accuracy $AB\epsilon_0$, for some scalar $0<\epsilon_0<1$ to be determined later.
This network has depth and size $\OO\left(\log_2\frac{1}{\epsilon_0}\right)$.

We construct the network $\widetilde{\Pi}$ that implements the multiplication $x_1x_2\ldots x_d$ as follows,
\begin{align*}
y_1 &= \widetilde{\pi}(x_1, x_2), &\vert y_1\vert &\leq M^2(1+\epsilon_0), \\
y_2 &= \widetilde{\pi}\left(y_1, x_3\right), &\vert y_2\vert &\leq M^3(1+\epsilon_0)^2, \\
y_3 &= \widetilde{\pi}\left(y_2, x_4\right), &\vert y_3\vert &\leq M^4(1+\epsilon_0)^3, \\
& \vdots && \vdots \\
y_{d-1} &= \widetilde{\pi}\left(y_{d-2}, x_d\right), &\vert y_{d-1}\vert &\leq M^{d}(1+\epsilon_0)^{d-1},
\end{align*}
and by setting $\widetilde{\Pi}(x_1,\ldots,x_d) = y_{d-1}$. 

The network $\widetilde{\Pi}$ has accuracy
\begin{align*}
\vert \widetilde{\Pi}(x_1,\ldots,x_d)-x_1\ldots x_d\vert 
& \leq \vert y_{d-1} - y_{d-2}x_d\vert + \vert x_d\vert\vert y_{d-2} - y_{d-3}x_{d-1}\vert + \ldots + \vert x_dx_{d-1}\ldots x_5\vert\vert y_3 - y_2x_4\vert \\
& \quad + \vert x_dx_{d-1}\ldots x_4\vert\vert y_2 - y_1x_3\vert + \vert x_dx_{d-1}\ldots x_3\vert\vert y_1 - x_1x_2\vert, \\
& < M^d(1+\epsilon_0)^{d-2}\epsilon_0 + M^d(1+\epsilon_0)^{d-3}\epsilon_0 + \ldots + M^d(1+\epsilon_0)^2 \\
& \quad + M^d(1+\epsilon_0) + M^d\epsilon_0, \\
& < dM^d(1+\epsilon_0)^d\epsilon_0 \quad \text{(crude estimate)}.
\end{align*}
We choose $\epsilon_0=\epsilon/(dM^de)$ to obtain accuracy $\epsilon$.

The depth and the size of the resulting network are equal to $(d-1)$ times the depth and size of the network defined at the beginning of the proof.
With accuracy $\epsilon_0$ defined above, this gives depth and size
\begin{align*}
\OO\left(d\log_2\frac{dM^de}{\epsilon}\right) = \OO\left(d\log_2\frac{d}{\epsilon} + d^2\log_2M\right).
\end{align*}
The proof is complete.
\end{proof}

The network of Proposition \ref{prop:multiplication} computes $x_1\ldots x_d$ as well as all the intermediate products $x_1\ldots x_k$, $2\leq k\leq d-1$, to the same accuracy $\epsilon$.
This allows us to prove the following result about polynomials, similar to \cite[Thm.~2]{liang2016}.

%%%%%%%%%%%%%%%%%%%%%%%%%%%%%%%%%%%%%%%%%%%%%%%%%%%%%%%%%%%%%%%%%%%%%%%%%%%
\begin{proposition}[Polynomials]\label{prop:monomial}
For any scalar $M\geq1$, $C\geq0$ and $0<\epsilon<1$, any integer $n\geq2$, and any polynomial $p_n$ of degree $n$ with input $x\in[-M,M]$ of the form
\begin{align*}
p_n(x)=\sum_{k=0}^nc_kx^k, \quad\underset{0\leq k\leq n}{\mrm{max}}\vert c_k\vert\leq C,
\end{align*}
there is a deep ReLU network $\wpp_n$ with inputs $(x_1,\ldots,x_n)\in[-M,M]^n$, that has depth
\begin{align*}
L=\OO\left(n\log_2\frac{Cn}{\epsilon} + n^2\log_2M\right),
\end{align*}
and size 
\begin{align*}
W=\OO\left(n\log_2\frac{Cn}{\epsilon} + n^2\log_2M\right),
\end{align*}
such that
\begin{align*}
\left\Vert\wpp_n(x,\ldots,x) - p_n(x)\right\Vert_{L^\infty([-M,M])}\leq\epsilon.
\end{align*}
\end{proposition} 

\begin{proof}
Let $M\geq 1$, $C\geq0$ and $0<\epsilon<1$ be three scalars, $n\geq 2$ an integer, and consider a polynomial
\begin{align*}
p_n(x)=\sum_{k=0}^nc_kx^k, \quad\underset{0\leq k\leq n}{\mrm{max}}\vert c_k\vert\leq C.
\end{align*}
We construct $\wpp(x_1,\ldots,x_n)$ as follows,
\begin{align*}
\wpp_n(x_1,\ldots,x_n) = c_0 + c_1x_1 + \sum_{k=2}^nc_ky_{k-1}(x_1,\ldots,x_k),
\end{align*}
where $y_{k-1}(x_1,\ldots,x_k)$ approximates $x_1\ldots x_k$ with the network of Proposition \ref{prop:multiplication} to accuracy $0<\epsilon_0<1$ to be determined later.
(Note that when the inputs are the same $y_{k-1}(x,\ldots,x)$ approximates $x^k$.)

The network $\wpp_n$ has accuracy
\begin{align*}
\vert\wpp_n(x,\ldots,x) - p_n(x)\vert 
& \leq C\sum_{k=2}^n\vert y_{k-1}(x,\ldots,x) - x^k\vert < nC\epsilon_0.
\end{align*}
We choose $\epsilon_0=\epsilon/(Cn)$ to obtain accuracy $\epsilon$.

The resulting network has depth and size
\begin{align*}
\OO\left(n\log_2\frac{Cn^2M^n}{\epsilon}\right) = \OO\left(n\log_2\frac{Cn}{\epsilon} + n^2\log_2M\right).
\end{align*}
The proof is complete.
\end{proof}

The Chebyshev polynomials of the first kind play a central role in approximation theory \cite{trefethen2013}. 
They are defined on the interval $[-1,1]$ via the three-term recurrence relation
\begin{align}
T_n(x) = 2xT_{n-1}(x) - T_{n-2}(x), \quad n\geq 2,
\label{eq:recurrence}
\end{align}
with $T_0=1$ and $T_1(x) = x$.
We show next how deep ReLU networks can efficiently implement Chebyshev polynomials, using the recurrence \eqref{eq:recurrence}.

%%%%%%%%%%%%%%%%%%%%%%%%%%%%%%%%%%%%%%%%%%%%%%%%%%%%%%%%%%%%%%%%%%%%%%%%%%%
\begin{proposition}[Chebyshev polynomials]\label{prop:chebyshev}
For any scalar $0<\epsilon<1$, any integer $n\geq2$ and any Chebyshev polynomial $T_n$ of degree $n$ with input $x\in[-1,1]$, there is a deep ReLU network $\wTT_n$ with inputs $(x_1,\ldots,x_n)\in[-1,1]^n$, that has depth 
\begin{align*}
L=\OO\left(n\log_2\frac{n}{\epsilon} + n^2\right),
\end{align*}
and size 
\begin{align*}
W=\OO\left(n\log_2\frac{n}{\epsilon} + n^2\right),
\end{align*}
such that
\begin{align*}
\left\Vert\wTT_n(x,\ldots,x) - T_n(x)\right\Vert_{L^\infty([-1,1])}\leq\epsilon.
\end{align*}
\end{proposition} 

\begin{proof}
Let $0<\epsilon<1$ be a scalar and $n\geq2$ be an integer.
For any scalar $A\geq 1$ and $B\geq 1$, let us call $\widetilde{\pi}$ the network of Proposition \eqref{prop:yarotsky} that implements the multiplication $xy$, $x\in[-A,A]$, $y\in[-B,B]$, with accuracy $AB\epsilon_0$ for some scalar $0<\epsilon_0<1$ to be determined later.
This network has depth and size $\OO\left(\log_2\frac{1}{\epsilon_0}\right)$.

We construct the network $\wTT_n$ that approximates $T_n(x)$ as follows,
\begin{align*}
\wTT_0 & = 1, & \vert\wTT_0\vert & \leq1, \\
\wTT_1(x) & = x, & \vert\wTT_1\vert & \leq1, \\
\wTT_2(x,x) & = 2\widetilde{\pi}(x,\wTT_1) - \wTT_0, & \vert\wTT_2\vert & <(1+\epsilon_0)^2, \\
\wTT_3(x,x,x) & = 2\widetilde{\pi}(x,\wTT_2) - \wTT_1, & \vert\wTT_3\vert & <3(1+\epsilon_0)^3, \\
& \vdots && \vdots \\
\wTT_n(x,\ldots,x) & = 2\widetilde{\pi}(x,\wTT_{n-1}) - \wTT_{n-2}, & \vert\wTT_n\vert & <3^{n-2}(1+\epsilon_0)^n.
\end{align*}

Let us now estimate the accuracy $e_n$ of the network $\wTT_n(x,\ldots,x)$, 
where $e_n=\vert\wTT_n(x,\ldots,x) - T_n(x)\vert$. We have
\begin{align*}
e_n  & = \vert2\widetilde{\pi}(x,\wTT_{n-1}) - \wTT_{n-2} - 2xT_{n-1} + T_{n-2}\vert, \\
& \leq 2\vert\widetilde{\pi}(x, \wTT_{n-1}) - x\wTT_{n-1}\vert + 2\vert x\vert\vert\wTT_{n-1}- T_{n-1}\vert + e_{n-2}, \\
& \leq 2\epsilon_0\vert \wTT_{n-1}\vert\ + 2e_{n-1} + e_{n-2}, \\
& < 2\epsilon_03^{n-3}(1+\epsilon_0)^{n-1} + 2e_{n-1} + e_{n-2}, \\
%& \leq 2^{n-1}(1+\epsilon_0)^n\epsilon_0 + 2e_{n-1} + e_{n-2}, \\
& < n4^n(1+\epsilon_0)^n\epsilon_0 \quad \text{(crude estimate)}.
\end{align*}
We choose $\epsilon_0=\epsilon/(n4^ne)$ to obtain accuracy $\epsilon$.

The depth and the size of the resulting network are equal to $(n+1)$ times the depth and size of the network defined at the beginning of the proof.
With accuracy $\epsilon_0$ defined above, this gives depth and size
\begin{align*}
\OO\left(n\log_2\frac{n4^ne}{\epsilon}\right) = \OO\left(n\log_2\frac{n}{\epsilon} + n^2\right).
\end{align*}
The proof is complete.
\end{proof}

Note that we could have proven Proposition \ref{prop:chebyshev} using Proposition \ref{prop:monomial} and an estimate for the size $C$ of the coefficients of the expansion of $T_n$ in the monomial basis. (The leading term of $T_n$ grows like $2^{n-1}$, while the other terms grow at most like $c^n$, for some $c<4$.)

Since Proposition \ref{prop:chebyshev} implements $T_n$, as well as the intermediate $T_k$'s, $0\leq k\leq n-1$, to the same accuracy $\epsilon$, we have the following result about truncated Chebyshev series.

%%%%%%%%%%%%%%%%%%%%%%%%%%%%%%%%%%%%%%%%%%%%%%%%%%%%%%%%%%%%%%%%%%%%%%%%%%%
\begin{proposition}[Truncated Chebyshev series]\label{prop:series}
For any scalar $C>0$ and $0<\epsilon<1$, any integer $n\geq2$, and any truncated Chebyshev series $f_n$ of degree $n$ with input $x\in[-1,1]$, {real} coefficients $c_k$'s, $1\leq k \leq n$, of the form
\begin{align*}
f_n(x)=\sum_{k=0}^nc_kT_k(x), \quad\underset{0\leq k\leq n}{\mrm{max}}\vert c_k\vert\leq C,
\end{align*}
there is a deep ReLU network $\wff_n$ with inputs $(x_1,\ldots,x_n)\in[-1,1]^n$, that has depth 
\begin{align*}
L=\OO\left(n\log_2\frac{Cn}{\epsilon} + n^2\right),
\end{align*} 
and size 
\begin{align*}
W=\OO\left(n\log_2\frac{Cn}{\epsilon} + n^2\right),
\end{align*}
such that
\begin{align*}
\left\Vert\wff_n(x,\ldots,x) - f_n(x)\right\Vert_{L^\infty([-1,1])}\leq\epsilon.
\end{align*}
\end{proposition} 

\begin{proof}
Let $C\geq0$ and $0<\epsilon<1$ be two scalars, $n\geq 2$ an integer. Consider a truncated Chebyshev series of the form
\begin{align*}
f_n(x)=\sum_{k=0}^nc_kT_k(x), \quad\underset{0\leq k\leq n}{\mrm{max}}\vert c_k\vert\leq C,
\end{align*}
for some real coefficients $c_k$'s, $1\leq k\leq n$.

We construct $\wff_n$ as follows,
\begin{align}
\wff_n(x_1,\ldots,x_n) = c_0 + c_1x_1 + \sum_{k=2}^nc_k\wTT_k(x_1,\ldots,x_k),
\end{align}
where $\wTT_k$ approximates $T_k$ with the network of Proposition \ref{prop:chebyshev} to accuracy $0<\epsilon_0<1$ to be determined later.

The network $\wff_n$ has accuracy
\begin{align*}
\vert\wff_n(x,\ldots,x) - f_n(x)\vert 
& \leq C\sum_{k=2}^n\vert \wTT_k - T_k\vert < nC\epsilon_0.
\end{align*}
We choose $\epsilon_0=\epsilon/(Cn)$ to obtain accuracy $\epsilon$.

The resulting network has depth
\begin{align*}
\OO\left(n\log_2\frac{Cn^2}{\epsilon} + n^2\right) = \OO\left(n\log_2\frac{Cn}{\epsilon} + n^2\right),
\end{align*}
and size
\begin{align*}
\OO\left(n\log_2\frac{Cn^2}{\epsilon} + n^2\right) = \OO\left(n\log_2\frac{Cn}{\epsilon} + n^2\right).
\end{align*}
The proof is complete.
\end{proof}

Chebyshev series lie at the heart of approximation theory.
In particular, it is possible to show that Lipschitz continuous functions $f$ with input $x\in[-M,M]$ have a unique absolutely and uniformly convergent Chebyshev series, and we write $f(x)=\sum_{k=0}^{\infty}c_kT_k(x/M)$~\cite[Thm.~3.1]{trefethen2013}. For analytic functions, the truncated Chebyshev series defined as $f_n(x)=\sum_{k=0}^{n}c_kT_k(x/M)$ are \textit{exponentially accurate} approximations \cite[Thm.~8.2]{trefethen2013}.

More precisely, for some scalars $M\geq1$ and $s>1$, let us define
\begin{align*}
a_s^M=M\frac{s+s^{-1}}{2}, \quad b_s^M = M\frac{s-s^{-1}}{2},
\end{align*}
and the \textit{Bernstein $s$-ellipse scaled to $[-M,M]$},
\begin{align*}
E_s^M = \left\{x+iy\in\C\,:\,\frac{x^2}{(a_s^M)^2}+\frac{y^2}{(b_s^M)^2}=1\right\}.
\end{align*}
(It has foci $\sqrt{(a_r^M)^2-(b_r^M)^2}=\pm M$, semi-major axis $a_s^M$ and semi-minor axis $b_s^M$.)
If a function $f$ is analytic on the interval $[-M,M]$, and analytically continuable to the ellipse $E_s^M$, where it satisfies $\vert f(x)\vert<C_f$, for some $C_f>0$, then, for each $n\geq0$, the truncated Chebyshev series $f_n$ satisfies
\begin{align}
\left\Vert f_n - f\right\Vert_{L^\infty([-M,M])} \leq \frac{2C_fs^{-n}}{s-1}.
\label{eq:trefethen}
\end{align}

Using Proposition \ref{prop:series} and Equation \eqref{eq:trefethen}, we prove a result about the approximation of analytic functions by deep ReLU networks. 

%%%%%%%%%%%%%%%%%%%%%%%%%%%%%%%%%%%%%%%%%%%%%%%%%%%%%%%%%%%%%%%%%%%%%%%%%%%
\begin{theorem}[Deep networks for analytic functions]\label{thm:analytic}
For any scalar $M\geq1$, $s>1$, $C_f>0$ and $0<\epsilon<1$, and any {real-valued} analytic function $f$ with input $x\in[-M,M]$ that is analytically continuable to the open ellipse $E_s^M$, where it satisfies $\vert f(x)\vert\leq C_f$, there is a deep ReLU network $\wff_n$ with inputs $(x_1,\ldots,x_n)\in[-M,M]^n$, that has depth 
\begin{align*}
L=\OO\left(\frac{1}{\log_2^2s}\log_2^2\frac{C_f}{\epsilon}\right),
\end{align*}
and size 
\begin{align*}
W=\OO\left(\frac{1}{\log_2^2s}\log_2^2\frac{C_f}{\epsilon}\right),
\end{align*}
such that
\begin{align*}
\left\Vert\wff_n(x,\ldots,x) - f(x)\right\Vert_{L^\infty([-M,M])}\leq\epsilon.
\end{align*}
\end{theorem} 

\begin{proof}
Let $M\geq1$, $s>1$, $C_f>0$ and $0<\epsilon<1$ be four scalars, and $f$ be an analytic function defined on $[-M,M]$ that is analytically continuable to the open Bernstein $s$-ellipse $E_{s}^M$, where it satisfies $\vert f(x)\vert\leq C_f$. We first approximate $f$ by a truncated Chebyshev series $f_n$, and then approximate $f_n$ by a deep ReLU network $\wff_n$ using Proposition \ref{prop:series}.

Since $f$ is analytic in the open Bernstein $s$-ellipse $E_s^M$ then, for any integer $n\geq2$,
\begin{align*}
\left\Vert f_n(x)-f(x)\right\Vert_{L^\infty([-M,M])}\leq\frac{2C_fs^{-n}}{s-1} = \OO\left(C_fs^{-n}\right).
\end{align*}
Therefore, if we take $n=\OO\left(\frac{1}{\log_2s}\log_2\frac{2C_f}{\epsilon}\right)$, then the above term is bounded by $\epsilon/2$.

Let us now approximate $f_n$ by a deep ReLU network $\wff_n$. We first write
\begin{align*}
f_n(x) = \sum_{k=0}^nc_kT_k\left(\frac{x}{M}\right), 
\end{align*}
with
\begin{align}
\underset{0\leq k\leq n}{\mrm{max}}\vert c_k\vert = \OO\left(C_fs\right),\;\text{via \cite[Thm.~8.1]{trefethen2013}}.
\label{eq:coefficiens}
\end{align}
We then define our network $\wff_n$ as in Proposition \ref{prop:series}, with extra scaling $x/M$, and such that
\begin{align*}
\vert\wff_n(x,\ldots,x) - f_n(x)\vert\leq\frac{\epsilon}{2}.
\end{align*}
This yields
\begin{align*}
\vert\wff_n(x,\ldots,x) - f(x)\vert 
& \leq \vert\wff_n(x,\ldots,x) - f_n(x)\vert + \vert f_n(x) - f(x)\vert \leq \frac{\epsilon}{2} + \frac{\epsilon}{2} = \epsilon.
\end{align*}

To compute the depth and the size of the resulting network, we note that (i) the extra scaling $x/M$ adds a layer and increases the size by $\OO(n)$, (ii) the coefficients satisfy Equation \eqref{eq:coefficiens}, and (iii) the truncated series was computed to accuracy $\epsilon/2$. Therefore, the network $\wff_n(x,\ldots,x)$ has depth
\begin{align*}
\OO\left(n\log_2\frac{2C_fsn}{\epsilon} + n^2 + 1\right) = \OO\left(n\log_2\frac{2C_fsn}{\epsilon} + n^2\right),
\end{align*}
and size
\begin{align*}
\OO\left(n\log_2\frac{2C_fsn}{\epsilon} + n^2 + n\right) = \OO\left(n\log_2\frac{2C_fsn}{\epsilon} + n^2\right) .
\end{align*}
Using $n=\OO\left(\frac{1}{\log_2s}\log_2\frac{2C_f}{\epsilon}\right)$, this gives depth and size
\begin{align*}
\OO\left(\left(\frac{1}{\log_2s}\log_2\frac{2C_f}{\epsilon}\right)\log_2\left(\frac{2C_fs}{\epsilon}\frac{1}{\log_2s}\log_2\frac{2C_f}{\epsilon}\right) + \frac{1}{\log_2^2s}\log_2^2\frac{2C_f}{\epsilon}\right)
= \OO\left(\frac{1}{\log_2^2s}\log_2^2\frac{C_f}{\epsilon}\right).
\end{align*}
The proof is complete.
\end{proof}

Our result below could be generalized to multiple dimensions, which would be interesting future work. 
In \cite{e2018}, it was shown that deep ReLU networks can approximate multivariate analytic functions with exponential convergence, a result similar to our theorem above.
However, we would like to emphasize that it is not possible to directly apply the result in \cite{e2018} to prove our main theorem in Section 3, because it is only valid on an open interval contained in $[-1,1]$, instead of an arbitrary closed interval $[-M,M]$. 
 
Let us now highlight that, in general, the constants $s$ and $C_f$ depend on $M$.
Let us look at two examples, a function with a singularity on the imaginary axis and an \textit{entire} function (\textit{i.e.}, a function that is analytic over the whole complex plane).
A typical example of an analytic function with singularities on the imaginary axis is the Runge-like function $f(x) = 1/(1 + \frac{x^2}{\beta^2})$, $\beta>1$, whose singularities are located at $x=\pm i\beta$.
The function $f$ is analytic on the interval $[-M,M]$ and analytically continuable to the open Bernstein $s$-ellipse $E_s^M$ with
\begin{align*}
s(M) = \frac{\sqrt{(4M^2-2)r^2 + r^4 +1} + r^2 - 1}{2Mr}
\end{align*}
and $r = \beta + \sqrt{\beta^2 + 1}$. Since $f$ increases along the imaginary axis we may take
\begin{align*}
C_f(M) = f\left(M\frac{s(M) - s(M)^{-1}}{2}\right).
\end{align*}

The complex exponential $f(x)=e^{ix}$ is an entire function.
Hence, any $s>1$ works but $C_f(s,M)$ must grow with $s$ and $M$. 
As $f$ increases along the imaginary axis we may choose
\begin{align}
C_f(s,M) = f\left(M\frac{s - s^{-1}}{2}\right) = e^{M\frac{s-s^{-1}}{2}}.
\label{eq:exponential}
\end{align}
In this case the network of Thm.~\ref{thm:analytic} has depth and size
\begin{align*}
\OO\left(\frac{1}{\log_2^2s}\left(M\frac{s-s^{-1}}{2} + \log_2\frac{1}{\epsilon}\right)^2\right).
\end{align*}

We would also like to mention that the ReLU activation function is not an optimal choice for constructing neural networks to approximate smooth functions. 
For example, Thm.~2.3 of \cite{mhaskar1996} shows that one-hidden-layer shallow networks with $\OO\left(\log\left(\frac{1}{\epsilon}\right)\right)$ parameters can approximate analytic functions with $\epsilon$ accuracy when a smooth activation function is used. 
The disadvantage of the ReLU activation function in this scenario is not unexpected since it is not a natural choice to use a function that is not differentiable to approximate a smooth function. 
However, from the point of view of deep learning and optimization, ReLU is a much better choice \cite{e2018}. 
The study in this paper should be regarded as a complement to existing approximation theory, using a more modern approach.

%%%%%%%%%%%%%%%%%%%%%%%%%%%%%%%%%%%%%%%%%%%%%%%%%%%%%%%%%%%%%%%%%%%%%%%%%%%
\section{Approximation of generalized bandlimited functions by deep ReLU networks}\label{sec:theorem}

A famous theorem of Carath\'{e}odory states that if a point $x\in\R^d$ lies in the the convex hull of a set $P$, then $x$ can be written as the convex combination of at most $d+1$ points in $P$.
Maurey's theorem \cite{pisier1981} is an extension of Carath\'{e}odory's result to the infinite-dimensional case.
It was used in the context of shallow network approximations by Barron in 1993 \cite{barron1993}. 
We recall Maurey's theorem below.

%%%%%%%%%%%%%%%%%%%%%%%%%%%%%%%%%%%%%%%%%%%%%%%%%%%%%%%%%%%%%%%%%%%%%%%%%%%
\begin{theorem}[Maurey's theorem]\label{thm:maurey}
Let $H$ be a Hilbert space with norm $\Vert\cdot\Vert$. 
Suppose there exists $G\subset H$ such that for every $g\in G$, $\Vert g\Vert\leq b$ for some $b>0$. 
Then, for every $f$ in the convex hull of $G$ and every integer $n\geq 1$, there is a $f_n$ in the convex hull of $n$ points in $G$ and a constant $c>b^2-\Vert f\Vert^2$ such that $\Vert f - f_n\Vert^2\leq \frac{c}{n}$.
\end{theorem}

We are now ready to prove our main theorem about the approximation of generalized bandlimited functions of the form \eqref{eq:bandlimited} by deep ReLU networks. {Let us first define a Hilbert space of such functions.}

\begin{definition}[Generalized bandlimited functions]
{Let $d\geq 2$ be an integer, $M\geq 1$ be a scalar, and $B=[0,1]^d$. Suppose $K:\R\rightarrow\C$ is analytic and bounded by a constant $D_K\in(0,1]$ on $[-dM,dM]$, and that $K$ satisfies the assumption of Thm.~$\ref{thm:analytic}$ for some $s>1$ and $C_K>0$. We define the Hilbert space $\mathcal{H}_{K,M}(B)$ of generalized bandlimited functions via}
\begin{align*}
\mathcal{H}_{K,M}(B)=\left\{ f(\bs{x})=\int_{[-M,M]^d}F(\bs{w})K(\bs{w}\cdot\bs{x})d\bs{w}\; \bigg\arrowvert \; F:[-M,M]^d\rightarrow\C\text{ is in }L^2([-M,M]^d) \right\},
\end{align*}
with inner product $\langle f, g\rangle_{\mathcal{H}_{K,M}(B)}:=\int_{[-M,M]^d} F_f(\bs{w}) \overline{F}_g(\bs{w}) d\bs{w}$ and norm $\|f\|_{\mathcal{H}_{K,M}(B)}:=\|F_f\|_{L^2([-M,M]^d)}$, where
\begin{align*}
 F_f=\arg\min_{F\in S_f} \|F\|_{L^2([-M,M]^d)}, \quad S_f=\bigg\{F\;\bigg\arrowvert\; f(\bs{x})=\int_{[-M,M]^d}F(\bs{w})K(\bs{w}\cdot\bs{x})d\bs{w}\bigg\}.
\end{align*}

\end{definition}
{Note that}
\begin{align*}
|f(\bs{x})|\leq D_K \int_{[-M,M]^d} |F_f(\bs{w})|d\bs{w} \leq (2M)^{d/2}\|F_f\|_{L^2([-M,M]^d)} = (2M)^{d/2} \|f\|_{\mathcal{H}_{K,M}(B)}.
\end{align*}
{The above inequality shows that if we consider an evaluation functional $L_{\bs{x}}$ defined on $\mathcal{H}_{K,M}(B)$ as follows,}
\begin{align*}
f(\bs{x})=L_{\bs{x}}(f):=\int_{[-M,M]^d}F_f(\bs{w})K(\bs{w}\cdot\bs{x})d\bs{w},
\end{align*}
{then $L_{\bs{x}}$ is bounded on $\mathcal{H}_{K,M}(B)$. Hence, $\mathcal{H}_{K,M}(B)$ is a RKHS; it contains the space of bandlimited functions, which correponds to $K(t)=e^{it}$. For simplicity, we will use $F$ instead of $F_f$ for $f\in\mathcal{H}_{K,M}(B)$, when the dependency on $f$ is clear.}

%%%%%%%%%%%%%%%%%%%%%%%%%%%%%%%%%%%%%%%%%%%%%%%%%%%%%%%%%%%%%%%%%%%%%%%%%%%
\begin{theorem}[Deep networks for $\mathcal{H}_{K,M}$]\label{thm:bandlimited}
{Suppose $f$ is an arbitrary real-valued function in $\mathcal{H}_{K,M}(B)$, for some function $K$, scalars $M\geq 1$, $s>1$ and $C_K>0$, and integer $d\geq 2$. Let us assume that $\int_{\R^d}\vert F(\bs{w})\vert d\bs{w} = \int_{[-M,M]^d}\vert F(\bs{w})\vert d\bs{w} = C_F$. Then, for any measure $\mu$ and any scalar $0<\epsilon<1$, there exists a deep ReLU network $\wff$ with inputs $\bs{x}\in B=[0,1]^d$, that has depth}
\begin{align*}
L=\OO\left(\frac{1}{\log_2^2s}\log_2^2\frac{C_FC_K\sqrt{\mu(B)}}{\epsilon}\right),
\end{align*}
and size 
\begin{align*}
W=\OO\left(\frac{C_F^2\mu(B)}{\epsilon^2\log_2^2s}\log_2^2\frac{C_FC_K\sqrt{\mu(B)}}{\epsilon}\right),
\end{align*}
such that
\begin{align*}
\left\Vert\wff - f\right\Vert_{L^2(\mu, B)} = \sqrt{\int_B\vert\wff(\bs{x}) - f(\bs{x})\vert^2 d\mu(\bs{x})}\leq\epsilon.
\end{align*}
\end{theorem}

\begin{proof}
Let $f$ be an arbitrary function in $\mathcal{H}_{K,M}$, and $\mu$ be an arbitrary measure. Let $F(\bs{w})=\vert F(\bs{w})\vert e^{i\theta(\bs{w})}$.
Since $f$ is real-valued, we may write
\begin{align*}
f(\bs{x}) & = \mrm{Re}\,\Bigg(\int_{\R^d}F(\bs{w})K(\bs{w}\cdot\bs{x})d\bs{w}\Bigg), \\
& =  \mrm{Re}\,\Bigg(\int_{\R^d}C_F e^{i\theta(\bs{w})}K(\bs{w}\cdot\bs{x})\frac{\vert F(\bs{w})\vert}{C_F} d\bs{w}\Bigg), \\
& = \int_{[-M,M]^d}C_F\Bigg[\cos(\theta(\bs{w}))K_R(\bs{w}\cdot\bs{x})-\sin(\theta(\bs{w}))K_I(\bs{w}\cdot\bs{x})\Bigg]\frac{\vert F(\bs{w})\vert}{C_F} d\bs{w},
\end{align*}
{where $K_R(\bs{w}\cdot\bs{x})=\mrm{Re}(K(\bs{w}\cdot\bs{x}))$ and $K_I(\bs{w}\cdot\bs{x})=\mrm{Im}(K(\bs{w}\cdot\bs{x}))$.} The integral above represents $f$ as an infinite convex combination of functions in the set
\begin{align*}
G_{K,M}  = \Big\{\gamma\big[\cos(\beta)\mrm{Re}(K(\bs{w}\cdot\bs{x}))-\sin(\beta)\mrm{Im}(K(\bs{w}\cdot\bs{x}))\big],\,\vert\gamma\vert\leq C_F,\,\beta\in\R,\,\bs{w}\in[-M,M]^d\Big\}.
\end{align*}
Therefore, $f$ is in the closure of the convex hull of $G_{K,M}$.
Since functions in $G_{K,M}$ are bounded in the $L^2(\mu,B)$-norm by $2C_FD_K\sqrt{\mu(B)}\leq2C_F\sqrt{\mu(B)}$, Theorem \ref{thm:maurey} tells us that there exist real coefficients $b_j$'s and $\beta_j$'s such that\footnote{We use Theorem \ref{thm:maurey} with $b=2C_F\sqrt{\mu(B)}$, $c=b^2>b^2 - \Vert f\Vert^2$, and $\Vert\cdot\Vert=\Vert\cdot\Vert_{L^2(\mu, B)}$.}
\begin{align*}
f_{\epsilon_0}(\bs{x}) 
= \sum_{j=1}^{\ceil{1/\epsilon_0^2}}b_j\big[\cos(\beta_j)K_R(\bs{w}\cdot\bs{x}) - \sin(\beta_j)K_I(\bs{w}\cdot\bs{x})\big],\quad\sum_{j=1}^{\ceil{1/\epsilon_0^2}}\vert b_j\vert \leq C_F,
\end{align*}
for some $0<\epsilon_0<1$ to be determined later, such that
\begin{align*}
\left\Vert f_{\epsilon_0}(\bs{x}) - f(\bs{x})\right\Vert_{L^2(\mu, B)}\leq 2C_F\sqrt{\mu(B)}\epsilon_0.
\end{align*}

We now approximate $f_{\epsilon_0}(\bs{x})$ by a deep ReLU network $\wff(\bs{x})$. {Note that $K_R$ and $K_I$ are both analytic and satisfy the same assumptions as $K$. Using Theorem} \ref{thm:analytic}{, they can be approximated to accuracy $\epsilon_0$ using networks $\wKK_R$ and $\wKK_I$ of depth and size}
\begin{align*}
\OO\left(\frac{1}{\log_2^2s}\log_2^2\frac{C_K}{\epsilon_0}\right).
\end{align*}
We define the deep ReLU network $\wff(\bs{x})$ by
\begin{align*}
\wff(\bs{x}) = \sum_{j=1}^{\ceil{1/\epsilon_0^2}}b_j\big[\cos(\beta_j)\wKK_R(\bs{w}\cdot\bs{x}) - \sin(\beta_j)\wKK_I(\bs{w}\cdot\bs{x})\big].
\end{align*}
This network has depth $L=\OO\left(\frac{1}{\log_2^2s}\log_2^2\frac{C_K}{\epsilon_0}\right)$ and size $W=\OO\left(\frac{1}{\epsilon_0^2\log_2^2s}\log_2^2\frac{C_K}{\epsilon_0}\right)$, and
\begin{align*}
\vert \wff(\bs{x}) - f_{\epsilon_0}(\bs{x})\vert 
& \leq\sum_{j=1}^{\ceil{1/\epsilon_0^2}}\vert b_j\vert\vert\wKK_R(\bs{w}_j\cdot \bs{x})-K_R(\bs{w}_j\cdot\bs{x})\vert
+ \sum_{j=1}^{\ceil{1/\epsilon_0^2}}\vert b_j\vert\vert\wKK_I(\bs{w}_j\cdot \bs{x})-K_I(\bs{w}_j\cdot\bs{x})\vert\leq 2C_F\epsilon_0,
\end{align*}
which yields
\begin{align*}
\left\Vert\wff(\bs{x}) - f_{\epsilon_0}(\bs{x})\right\Vert_{L^2(\mu, B)}\leq 2C_F\sqrt{\mu(B)}\epsilon_0.
\end{align*}

The total approximation error satisfies
\begin{align*}
\left\Vert\wff(\bs{x}) - f(\bs{x})\right\Vert_{L^2(\mu, B)}\leq 4C_F\sqrt{\mu(B)}\epsilon_0.
\end{align*}
We take 
\begin{align*}
\epsilon_0=\frac{\epsilon}{4C_F\sqrt{\mu(B)}}
\end{align*}
to complete the proof.
\end{proof}

Let us end this section with comments on the constants $C_F$, $C_K$ and $\mu(B)$; we start with $C_F$.
If $F$ is a mollifier then $C_F=1$, whereas if $F$ is a normal distribution truncated to $[-M,M]^d$ then $C_F<1$.
In general, however, $C_F$ might grow algebraically or exponentially with the dimension $d$.

We continue with $C_K$. 
Consider for example the complex exponential kernel $K(t)=e^{it}$, $t\in[-dM,dM]$.
Equation \ref{eq:exponential} yields
\begin{align*}
C_K(s,dM)=e^{dM\frac{s-s^{-1}}{2}}, \quad \text{for any $s>1$}.
\end{align*}
The resulting network to approximate a function to accuracy $\epsilon$ in the $L^2(\mu,B)$-norm with such a kernel 
has depth
\begin{align*}
L=\OO\left(\frac{1}{\log_2^2s}\left(dM\frac{s-s^{-1}}{2} + \log_2\frac{C_F\sqrt{\mu(B)}}{\epsilon}\right)^2\right),
\end{align*}
and size
\begin{align*}
W=\OO\left(\frac{C_F^2\mu(B)}{\epsilon^2\log_2^2s}\left(dM\frac{s-s^{-1}}{2} + \log_2\frac{C_F\sqrt{\mu(B)}}{\epsilon}\right)^2\right).
\end{align*}

We conclude with $\mu(B)$.
If $\mu$ is a probability measure, then $\mu(B)\leq1$ for any compact domain $B$.
If $\mu$ is Lebesgue measure, then $\mu(B)=1$ for $B=[0,1]^d$, but grows exponentially with the dimension $d$ if $B=[0,\ell]^d$, $\ell>1$. 
This is a common drawback in the approximation theory of neural networks for conquering the curse of dimensionality \cite{barron1993}. 

%%%%%%%%%%%%%%%%%%%%%%%%%%%%%%%%%%%%%%%%%%%%%%%%%%%%%%%%%%%%%%%%%%%%%%%%%%%
\section{Discussion}

We have proven new upper bounds for the approximation of bandlimited functions of the form \eqref{eq:bandlimited}, for which the curse of dimensionality is overcome.
Our proof is based on Maurey's theorem and on the ability of deep ReLU networks to approximate Chebyshev polynomials and analytic functions efficiently.

There are many ways in which this work could be profitably continued.
The space of bandlimited functions is a type of RKHS and therefore a possible extension would be to look at different types of RKHS. 
One could also relax the bandlimited assumption \eqref{eq:bandlimited}, e.g., to functions $F$ whose derivatives are rapidly decreasing.
In this case, the kernel $K$ could be approximated on the real line by Chebyshev polynomials on truncated intervals or Hermite polynomials.
The latter is another example of classical orthogonal polynomials, which can be represented by a three-term recurrence relation similar to \eqref{eq:recurrence} and efficiently implemented by deep ReLU networks.

Let us conclude this paper with a comment on deep versus shallow networks in the context of parallel computing.
Since the depth $L$ grows like $\OO\left(\log_2^2\frac{1}{\epsilon}\right)$ in Theorem \ref{thm:bandlimited}, the approximation accuracy for deep networks can be root-exponentially improved if $L$ increases. 
Hence, very deep networks are more efficient than shallow networks when both parallel computing efficiency and approximation efficiency are considered.
This is in contrast with the more general case of continuous functions, the approximation of which via very deep networks might be less attractive in terms of parallel computing \cite{shen2019}.

%%%%%%%%%%%%%%%%%%%%%%%%%%%%%%%%%%%%%%%%%%%%%%%%%%%%%%%%%%%%%%%%%%%%%%%%%%%
\section*{Acknowledgements}

We thank the members of the CM3 group at Columbia University and our colleague Mikael Slevinsky for fruitful discussions.
The first author is much indebted to former PhD supervisor Nick Trefethen for his inspirational contributions to numerical analysis and in particular to the field of approximation theory.

The research of the second author was partially supported by US NSF under the grant award DMS-1945029.

The research of the third author is supported in part by US NSF DMS-1719699 and the NSF TRIPODS program CCF-1704833.

%%%%%%%%%%%%%%%%%%%%%%%%%%%%%%%%%%%%%%%%%%%%%%%%%%%%%%%%%%%%%%%%%%%%%%%%%%%
\bibliography{references.bib}

%%%%%%%%%%%%%%%%%%%%%%%%%%%%%%%%%%%%%%%%%%%%%%%%%%%%%%%%%%%%%%%%%%%%%%%%%%%
\end{document}